\documentclass[11pt,a4paper]{article}
\usepackage{amsfonts}
\usepackage{amsmath}

\newtheorem{theorem}{Theorem}

\newtheorem{corollary}[theorem]{Corollary}

\newtheorem{lemma}[theorem]{Lemma}

\newtheorem{proposition}[theorem]{Proposition}

\newenvironment{proof}[1][Proof]{\noindent\textbf{#1.} }{\ \rule{0.5em}{0.5em}}
\input{tcilatex}
\begin{document}

\title{Isoperimetric inequalities for Bergman analytic content}
\date{}
\author{Stephen J. Gardiner, Marius Ghergu\ and Tomas Sj\"{o}din}
\maketitle

\begin{abstract}
The Bergman $p$-analytic content ($1\leq p<\infty $) of a planar domain $%
\Omega $ measures the $L^{p}(\Omega )$-distance between $\overline{z}$ and
the Bergman space $A^{p}(\Omega )$ of holomorphic functions. It has a
natural analogue in all dimensions which is formulated in terms of harmonic
vector fields. This paper investigates isoperimetric inequalities for
Bergman $p$-analytic content in terms of the St Venant functional for
torsional rigidity, and addresses the cases of equality with the upper and
lower bounds.
\end{abstract}

\section{Introduction}

\footnotetext{%
\noindent 2010 \textit{Mathematics Subject Classification } 31B05.
\par
\noindent \textit{Keywords: harmonic vector field, Bergman space,
isoperimetric inequality, torsional rigidity}}The \textit{Bergman }$p$%
\textit{-analytic content} ($1\leq p<\infty $) of a bounded planar domain $%
\Omega $ was introduced by Guadarrama and Khavinson \cite{GuKh}. It is
defined by the formula $\lambda _{A^{p}}(\Omega )=\inf_{f\in A^{p}(\Omega
)}\left\Vert \overline{z}-f\right\Vert _{p}$, where $\left\Vert \cdot
\right\Vert _{p}$ is the usual $L^{p}(\Omega )$-norm and $A^{p}(\Omega )$ is
the Bergman space of $L^{p}(\Omega )$-integrable holomorphic functions $f$
on $\Omega $. In the case where $p=2$, Fleeman and Khavinson \cite{FK}
showed that, for any simply connected domain $\Omega $ with piecewise smooth
boundary, 
\begin{equation*}
\sqrt{\rho (\Omega )}\leq \lambda _{A^{2}}(\Omega )\leq \frac{m(\Omega )}{%
\sqrt{2\pi }},
\end{equation*}%
where $\rho (\Omega )$ denotes the torsional rigidity of $\Omega $ and $m$
is Lebesgue measure. Subsequently, Fleeman and Lundberg \cite{FL} showed
that the left hand inequality is actually an equality for any bounded simply
connected domain, and this relationship has been further exploited by
Fleeman and Simanek \cite{FS}. Bell, Ferguson and Lundberg \cite{BFL}
established related inequalities concerning torsional rigidity and the norm
of the self-commutator of a Toeplitz operator. The limiting case of Bergman $%
p$-analytic content where $p=\infty $ is the notion of analytic content,
which has been studied for many years: see, for example, \cite{GaKh}, \cite%
{BK}, \cite{ABKT} for the case of the plane, and \cite{GK}, \cite{GGS} for
its extension to higher dimensions.

Rewriting $\lambda _{A^{p}}(\Omega )$ as $\inf_{\phi \in A^{p}(\Omega
)}\left\Vert z-\overline{\phi }\right\Vert _{p}$, we see that a natural
generalization to bounded domains $\Omega $ in Euclidean space $\mathbb{R}%
^{N}$ ($N\geq 2$) is given by%
\begin{equation*}
\lambda _{A_{p}}(\Omega )=\inf \{\left\Vert x-f\right\Vert _{\mathcal{L}%
_{p}}:f\in A_{p}(\Omega )\}\text{ \ \ \ }(1\leq p<\infty ),
\end{equation*}%
where $A_{p}(\Omega )$ denotes the space of \textit{harmonic vector fields }$%
f=(f_{1},...,f_{N})$ in $\mathcal{L}_{p}\cap C^{1}(\Omega )$,\textit{\ }%
\begin{equation*}
\mathcal{L}_{p}=\mathcal{L}_{p}(\Omega )=\left( L^{p}(\Omega )\right) ^{N},%
\text{\ \ \ }\Vert f\Vert _{\mathcal{L}_{p}}=\left( \int_{\Omega }\left\Vert
f\right\Vert ^{p}dm\right) ^{1/p}
\end{equation*}%
and\ $\left\Vert \cdot \right\Vert $ is the usual Euclidean norm on $\mathbb{%
R}^{N}$. Thus $f$ satisfies $\func{div}f=0$ and $\func{curl}f=0$, where the
latter condition means that 
\begin{equation*}
\frac{\partial f_{j}}{\partial x_{k}}-\frac{\partial f_{k}}{\partial x_{j}}=0%
\text{ \ \ for all }j,k\in \{1,...,N\}\text{\ on\ }\Omega .
\end{equation*}%
The gradient of any harmonic function is a harmonic vector field, and the
converse assertion is also true when $\Omega $ is simply connected. We will
assume from now on that $\Omega $ is smoothly bounded.

The purpose of this paper is to investigate isoperimetric inequalities for $%
\lambda _{A_{p}}(\Omega )$ $(1\leq p<\infty )$ in all dimensions, and to
examine the cases of equality with the upper and lower bounds (cf. Problem
3.4 of \cite{BK2}). We denote by $q$ the dual exponent of $p$, whence $%
1/p+1/q=1$ (or $q=\infty $ if $p=1$), and note that the dual space $\mathcal{%
L}_{p}^{\ast }$ can be identified with $\mathcal{L}_{q}$.\ When $q<\infty $
we denote by $W_{0}^{1,q}(\Omega )$ the closure of $C_{c}^{\infty }(\Omega )$
in the\ Sobolev space $W^{1,q}(\Omega )$; these are the functions in $%
W^{1,q}(\Omega )$ that have trace zero on $\partial \Omega $ (see Section
5.5 of \cite{Ev}). Since any function in $W^{1,\infty }(\Omega )$ has a
Lipschitz representative, it is natural to denote by $W_{0}^{1,\infty
}(\Omega )$ the subset of $W^{1,\infty }(\Omega )$ comprising those
functions which vanish on $\partial \Omega $. We define%
\begin{equation}
Q_{q}(\Omega )=\sup_{u\in W_{0}^{1,q}(\Omega )\backslash \{0\}}\frac{N}{%
\Vert \nabla u\Vert _{\mathcal{L}_{q}}}\int_{\Omega }u~dm\text{ \ \ \ }%
(1<q\leq \infty ).  \label{Qdef}
\end{equation}%
When $q<\infty $, the quantity $\left( Q_{q}(\Omega )\right) ^{q}$ is known
as the \textit{St Venant }$q$\textit{-functional} of $\Omega $. Its
relationship with the torsional rigidity $\rho (\Omega )$ will be discussed
in Section \ref{S4}.

We begin with the case $p=2$, where we can add the following to the results
of \cite{FK} and \cite{FL}.

\begin{theorem}
\label{Cor}If $\Omega \subset \mathbb{R}^{N}$ is a smoothly bounded domain,
then $\lambda _{A_{2}}(\Omega )=Q_{2}(\Omega )$.\ Further, $\lambda
_{A_{2}}(\Omega )=\sqrt{\rho (\Omega )}$ if and only if $\mathbb{R}%
^{N}\backslash \Omega $ is connected.
\end{theorem}

Next, we establish a lower bound for $\lambda _{A_{p}}(\Omega )$ for all $p$.

\begin{theorem}
\label{main1}If $\Omega \subset \mathbb{R}^{N}$ is a smoothly bounded domain
and $p\in \lbrack 1,\infty )$, then 
\begin{equation}
Q_{q}(\Omega )\leq \lambda _{A_{p}}(\Omega ).  \label{LAP}
\end{equation}%
Further, equality holds if and only if either \newline
(a) $p=2$, or \newline
(b) $\Omega $ is a ball or an annular region.
\end{theorem}

The case of equality above when $p\neq 2$ is a counterpart of a recent
result of Abanov, B\'{e}n\'{e}teau, Khavinson and Teodorescu \cite{ABKT}
concerning analytic content in the plane (that is, where $p=\infty $ and $%
N=2 $).

It remains to establish an upper bound for $\lambda _{A_{p}}(\Omega )$. Let $%
B(r)$ denote the open ball in $\mathbb{R}^{N}$\ of centre $0$ and radius $r$%
, and let $B=B(1)$. Further, let $r_{\Omega }>0$ be chosen so that $%
m(B(r_{\Omega }))=m(\Omega )$. Then, by the generalized Faber-Krahn
inequality (cf. \cite{BdPV}), we have $Q_{q}(\Omega )\leq Q_{q}(B(r_{\Omega
}))$. The result below is new in all dimensions.

\begin{theorem}
\label{main2}If $\Omega \subset \mathbb{R}^{N}$ is a smoothly bounded domain
and $p\in \lbrack 1,2]$, then 
\begin{equation}
\lambda _{A_{p}}(\Omega )\leq Q_{q}(B(r_{\Omega })).  \label{LAP1}
\end{equation}%
Further, equality holds if and only if $\Omega $ is a ball.
\end{theorem}

We will see later, in Proposition \ref{Qq}, that the upper bound in (\ref%
{LAP1}) is given explicitly by\textbf{\ }%
\begin{equation*}
Q_{q}(B(r))=\left( \dfrac{N}{N+p}m(B)\right) ^{1/p}r^{1+N/p}\ \ \ (1\leq
p<\infty ).
\end{equation*}%
Recent work of the authors \cite{GGS} shows that there is a harmonic
function $h$ on $\Omega $ satisfying $\sup_{\Omega }\left\Vert x-\nabla
h\right\Vert \leq r_{\Omega }$, whence $\lambda _{A_{p}}(\Omega )\leq
(m(B))^{1/p}r_{\Omega }^{1+N/p}$ for general $p$. We conjecture that balls
are always the extremal domains for (\ref{LAP1}); that is, the sharper
estimate of Theorem \ref{main2}, 
\begin{equation*}
\lambda _{A_{p}}(\Omega )\leq \left( \frac{N}{N+p}m(B)\right)
^{1/p}r_{\Omega }^{1+N/p},
\end{equation*}%
remains valid for all $p\in \lbrack 1,\infty )$.

Theorems \ref{main1} and \ref{main2} together yield the following
isoperimetric inequality for Bergman $p$-analytic content.

\begin{corollary}
If $\Omega \subset \mathbb{R}^{N}$ is a smoothly bounded domain and $p\in
\lbrack 1,2]$, then%
\begin{equation*}
Q_{q}(\Omega )\leq \lambda _{A_{p}}(\Omega )\leq Q_{q}(B(r_{\Omega })).
\end{equation*}
\end{corollary}

The remainder of the paper is devoted to proving the above results.

\section{Existence and uniqueness of extremal functions\label{S3}}

In the course of proving our results concerning $\lambda _{A_{p}}(\Omega )$
and $Q_{q}(\Omega )$, we are led to consider the related domain constants%
\begin{eqnarray*}
\lambda _{B_{p}}(\Omega ) &=&\inf \{\left\Vert x-f\right\Vert _{\mathcal{L}%
_{p}}:f\in B_{p}(\Omega )\}\text{,} \\
\lambda _{D_{p}}(\Omega ) &=&\inf \{\left\Vert x-f\right\Vert _{\mathcal{L}%
_{p}}:f\in D_{p}(\Omega )\},
\end{eqnarray*}%
where 
\begin{eqnarray*}
B_{p}(\Omega ) &=&\{\nabla h:h\in W^{1,p}(\Omega )\cap C^{2}(\Omega )\text{
and }\Delta h=0\text{ on }\Omega \}, \\
D_{p}(\Omega ) &=&\{f\in \mathcal{L}_{p}:\func{div}f=0\text{ on }\Omega 
\text{ in the sense of distributions}\}.
\end{eqnarray*}%
Since $B_{p}(\Omega )\subset A_{p}(\Omega )\subset D_{p}(\Omega ),$ we see
that 
\begin{equation}
\lambda _{D_{p}}(\Omega )\leq \lambda _{A_{p}}(\Omega )\leq \lambda
_{B_{p}}(\Omega ).  \label{Lineq}
\end{equation}%
In this section we will prove existence and uniqueness results concerning
the extremal functions for $Q_{q}(\Omega )$, $\lambda _{D_{p}}(\Omega )$, $%
\lambda _{B_{p}}(\Omega )$ and $\lambda _{A_{p}}(\Omega )$.

Let $\Delta _{q}$ denote the $q$\textit{-Laplacian,} given by $\Delta
_{q}u=\nabla \cdot \left( \left\Vert \nabla u\right\Vert ^{q-2}\nabla
u\right) $, where $1<q<\infty $. \ We define the $q$\textit{-torsion
function }$w_{q}$ on $\Omega $ to be the weak solution of%
\begin{equation}
\left\{ 
\begin{array}{cc}
-\Delta _{q}w_{q}=1 & \text{in }\Omega \\ 
w_{q}=0 & \text{on }\partial \Omega%
\end{array}%
\right. ,  \label{qlaplace}
\end{equation}%
and note from \cite{Lie} that $w_{q}\in C^{1}(\overline{\Omega })$. Further,
we define $w_{\infty }(x)=\mathrm{dist}(x,\partial \Omega )$.

\begin{proposition}
\label{Qq}Let $\Omega \subset \mathbb{R}^{N}$ be a smoothly bounded domain
and $p\in \lbrack 1,\infty )$. \newline
(i) There exists $u\in W_{0}^{1,q}(\Omega )$ such that 
\begin{equation}
Q_{q}(\Omega )=\frac{N}{\Vert \nabla u\Vert _{\mathcal{L}_{q}}}\int_{\Omega
}u~dm.  \label{Ppmax}
\end{equation}%
(ii) The functions $u\in W_{0}^{1,q}(\Omega )$ which satisfy (\ref{Ppmax})
are precisely the positive multiples of $w_{q}$.\newline
(iii) $Q_{q}(\Omega )=N\left( \int_{\Omega }w_{q}dm\right) ^{1/p}$. Further, 
$Q_{q}(\Omega )=N\Vert \nabla w_{q}\Vert _{\mathcal{L}_{q}}^{q-1}$ if $p>1$.%
\newline
(iv) $Q_{q}(B(r))=\left( \dfrac{N}{N+p}m(B)\right) ^{1/p}r^{1+N/p}.$
\end{proposition}

\begin{proof}
(i) We choose a maximizing sequence $(u_{j})$ for (\ref{Qdef}) such that $%
\Vert u_{j}\Vert _{W_{0}^{1,q}(\Omega )}=1$ for all $j$. (The quotient in (%
\ref{Qdef}) is unaffected when $u$ is multiplied by a positive constant.)

Firstly, we suppose that $p>1$. In view of the Banach-Alaoglu theorem we can
arrange, by taking a subsequence, that $(u_{j})$ converges weakly to some
non-zero function $u\in W_{0}^{1,q}(\Omega )$. Further, by the
Rellich-Kondrachov theorem (see, for example, Section 5.7 in \cite{Ev}), we
can arrange that $u_{j}\rightarrow u$ strongly in $L^{1}(\Omega )$. Clearly $%
\int_{\Omega }u~dm>0$. By the weak lower semicontinuity of the $\mathcal{L}%
_{q}$-norm, 
\begin{eqnarray*}
Q_{q}(\Omega ) &=&\lim_{j\rightarrow \infty }\frac{N}{\Vert \nabla
u_{j}\Vert _{\mathcal{L}_{q}}}\int_{\Omega }u_{j}~dm \\
&=&\frac{\lim_{j\rightarrow \infty }N\int_{\Omega }u_{j}~dm}{%
\lim_{j\rightarrow \infty }\Vert \nabla u_{j}\Vert _{\mathcal{L}_{q}}}\leq 
\frac{N\int_{\Omega }u~dm}{\Vert \nabla u\Vert _{\mathcal{L}_{q}}}\leq
Q_{q}(\Omega ),
\end{eqnarray*}%
and so equality holds throughout.

If $p=1$, whence $q=\infty $, then we instead appeal to the Arzel\`{a}%
-Ascoli theorem to see that there is a subsequence of $(u_{j})$ that
converges uniformly on $\Omega $, and make use of the fact that each $u_{j}$
can be represented by a Lipschitz function.

(ii) Suppose firstly that $p>1$. For any $\phi \in C_{c}^{\infty }(\Omega )$
we define 
\begin{equation*}
f(t)=\left( Q_{q}(\Omega )\right) ^{q}\int_{\Omega }\Vert \nabla (u+t\phi
)\Vert ^{q}dm-\left( N\int_{\Omega }(u+t\phi )dm\right) ^{q}\text{ \ \ }%
(t\in {\mathbb{R}}).
\end{equation*}%
Since $u$ is a maximizer for $Q_{q}(\Omega )$, we see that $\int_{\Omega
}u~dm>0$ and $f^{\prime }(0)=0$, whence%
\begin{equation*}
\left( Q_{q}(\Omega )\right) ^{q}\int_{\Omega }\left\Vert \nabla
u\right\Vert ^{q-2}\nabla u\cdot \nabla \phi ~dm-N^{q}\left( \int_{\Omega
}u~dm\right) ^{q-1}\int_{\Omega }\phi ~dm=0\text{ \ \ }\left( \phi \in
C_{c}^{\infty }(\Omega )\right) ,
\end{equation*}%
and so%
\begin{equation*}
\int_{\Omega }\left( \left( Q_{q}(\Omega )\right) ^{q}\Delta
_{q}u+N^{q}\left( \int_{\Omega }u~dm\right) ^{q-1}\right) \phi ~dm=0\text{ \
\ }\left( \phi \in C_{c}^{\infty }(\Omega )\right) .
\end{equation*}%
Thus $\Delta _{q}u$ is a negative constant in $\Omega $, and so $u$ is a
positive multiple of $w_{q}$.

Now let $p=1$, so that $q=\infty $. In the formula (\ref{Qdef}) we can
normalize to consider only those functions $u$ such that $\left\Vert \nabla
u\right\Vert _{\mathcal{L}_{\infty }}=1$, whence $u$ is majorized by the
Lipschitz function$\ w_{\infty }$. Further, the supremum can only be
attained among functions $u$ satisfying $\left\Vert \nabla u\right\Vert _{%
\mathcal{L}_{\infty }}=1$ by the function $w_{\infty }$. More generally, the
supremum can only be attained by a positive multiple of $w_{\infty }$.

(iii) If $p>1$, then we see from (\ref{qlaplace}) that%
\begin{equation*}
\int_{\Omega }w_{q}~dm=\int_{\Omega }w_{q}(-\Delta
_{q}w_{q})~dm=\int_{\Omega }\left\Vert \nabla w_{q}\right\Vert ^{q-2}\nabla
w_{q}\cdot \nabla w_{q}~dm=\left\Vert \nabla w_{q}\right\Vert _{\mathcal{L}%
_{q}}^{q}.
\end{equation*}%
Hence, by parts (i) and (ii), $Q_{q}(\Omega )=N\Vert \nabla w_{q}\Vert _{%
\mathcal{L}_{q}}^{q-1}=N\left( \int_{\Omega }w_{q}dm\right) ^{1/p}$.

If $p=1$, then it is immediate that $Q_{\infty }(\Omega )=N\int_{\Omega
}w_{\infty }dm$.

(iv) If $\Omega =B(r)$, then $w_{q}(x)$ is clearly a multiple of $x\mapsto
r^{p}-\left\Vert x\right\Vert ^{p}$. Letting $u(x)=\left( r^{p}-\left\Vert
x\right\Vert ^{p}\right) /p$, we have $\left\Vert \nabla u\right\Vert
=\left\Vert x\right\Vert ^{p-1}$ and%
\begin{equation*}
\Vert \nabla u\Vert _{\mathcal{L}_{q}}=\left( \frac{Nm(B)}{N+p}%
r^{N+p}\right) ^{1/q},\text{ \ \ }\int_{\Omega }u~dm=\frac{m(B)}{N+p}r^{N+p}.
\end{equation*}%
Thus, by parts (i) and (ii),%
\begin{equation*}
Q_{q}(B(r))=\frac{N}{\Vert \nabla u\Vert _{\mathcal{L}_{q}}}\int_{\Omega
}u~dm=\left( \frac{N}{N+p}m(B)\right) ^{1/p}r^{1+N/p}.
\end{equation*}
\end{proof}

\bigskip

As usual, we define%
\begin{equation*}
D_{p}(\Omega )^{\perp }=\left\{ g\in \mathcal{L}_{q}(\Omega ):\int_{\Omega
}f\cdot g~dm=0\text{ for all }f\in D_{p}(\Omega )\right\} ;
\end{equation*}%
$B_{p}(\Omega )^{\perp }$ and $A_{p}(\Omega )^{\perp }$ are defined
analogously.

\begin{proposition}
\label{Dp}Let $\Omega \subset \mathbb{R}^{N}$ be a smoothly bounded domain
and $p\in (1,\infty )$. \newline
(i) There exists $f_{0}\in D_{p}(\Omega )$ such that $\lambda
_{D_{p}}(\Omega )=\left\Vert x-f_{0}\right\Vert _{\mathcal{L}_{p}}$.\newline
(ii) This function $f_{0}$ satisfies $\left\Vert x-f_{0}\right\Vert
^{p-2}(x-f_{0})\in D_{p}(\Omega )^{\perp }$.\newline
(iii) There exists $u_{0}\in W_{0}^{1,q}(\Omega )$ such that $\nabla
u_{0}=-\left\Vert x-f_{0}\right\Vert ^{p-2}(x-f_{0})$.\newline
(iv) The function $u_{0}$ is a positive multiple of $w_{q}$, and\ 
\begin{equation}
\lambda _{D_{p}}(\Omega )=Q_{q}(\Omega )=\dfrac{-1}{\Vert \nabla w_{q}\Vert
_{\mathcal{L}_{q}}}\dint_{\Omega }(x-f)\cdot \nabla w_{q}~dm\ \ \ \ (f\in
D_{p}(\Omega )).  \label{QqDp}
\end{equation}
\end{proposition}

\begin{proof}
(i) We choose a sequence $(f_{j})$ in $D_{p}(\Omega )$ such that $\left\Vert
x-f_{j}\right\Vert _{\mathcal{L}_{p}}\rightarrow \lambda _{D_{p}}(\Omega )$.
By weak compactness we can arrange, by choosing a suitable subsequence, that 
$(f_{j})$ is weakly convergent to some $f_{0}$ in $\mathcal{L}_{p}$.
Further, $\func{div}f_{0}=0$ on $\Omega $ in the sense of distributions, so $%
f_{0}\in D_{p}(\Omega )$. Finally, $\lambda _{D_{p}}(\Omega )=\left\Vert
x-f_{0}\right\Vert _{\mathcal{L}_{p}}$ by the weak lower semicontinuity of
the norm.

(ii) For any $g\in D_{p}(\Omega )$ we can differentiate the function $%
t\mapsto \int_{\Omega }\left\Vert x-f_{0}-tg\right\Vert ^{p}dm$ and then set 
$t=0$ to see that 
\begin{equation*}
\int_{\Omega }\left\Vert x-f_{0}\right\Vert ^{p-2}(x-f_{0})\cdot g~dm=0.
\end{equation*}

(iii) If $f\in \mathcal{L}_{p}$, then by definition, 
\begin{equation}
f\in D_{p}(\Omega )\Leftrightarrow \int_{\Omega }f\cdot \nabla \phi ~dm=0%
\text{\ \ \ \ }(\phi \in C_{c}^{\infty }(\Omega )).  \label{defDp}
\end{equation}%
Hence%
\begin{equation}
D_{p}(\Omega )^{\perp }=\overline{\{\nabla \phi :\phi \in C_{c}^{\infty
}(\Omega )\}}^{\mathcal{L}_{q}(\Omega )}  \label{Dpperp}
\end{equation}%
since, if $g\in \mathcal{L}_{q}(\Omega )$ does not belong to the above
closure, the Hahn-Banach theorem would yield the existence of $f\in \mathcal{%
L}_{q}^{\ast }(\Omega )\equiv \mathcal{L}_{p}(\Omega )$ such that%
\begin{equation*}
\int_{\Omega }f\cdot g~dm=1,\text{ \ \ \ }\int_{\Omega }f\cdot \nabla \phi
~dm=0\text{ \ \ \ }(\phi \in C_{c}^{\infty }(\Omega )),
\end{equation*}%
whence $f\in D_{p}(\Omega )$\ and so $g\notin D_{p}(\Omega )^{\perp }$.

We claim next that%
\begin{equation}
D_{p}(\Omega )^{\perp }=\{\nabla u:u\in W_{0}^{1,q}(\Omega )\}.
\label{Dpperpb}
\end{equation}%
Clearly the right hand side of (\ref{Dpperpb}) is contained in the right
hand side of (\ref{Dpperp}). To see the reverse inclusion, let $(\phi _{k})$
be a sequence in $C_{c}^{\infty }(\Omega )$ such that $(\nabla \phi _{k})$
converges in $\mathcal{L}_{q}(\Omega )$. Then $(\phi _{k})$ is Cauchy in $%
\mathcal{L}_{q}(\Omega )$, by Poincar\'{e}'s inequality for $%
W_{0}^{1,q}(\Omega )$. It follows that $(\phi _{k})$ converges in $%
W_{0}^{1,q}(\Omega )$ to some function $u$ and $\lim_{k\rightarrow \infty
}\nabla \phi _{k}=\nabla u$. Hence (\ref{Dpperpb}) holds, and the desired
conclusion now follows from part (ii).

(iv) By the divergence theorem, (\ref{Dpperpb}) and H\"{o}lder's inequality,%
\begin{eqnarray*}
\dfrac{N}{\Vert \nabla u\Vert _{\mathcal{L}_{q}}}\dint_{\Omega }u~dm &=&%
\dfrac{-1}{\Vert \nabla u\Vert _{\mathcal{L}_{q}}}\dint_{\Omega }x\cdot
\nabla u~dm \\
&=&\dfrac{-1}{\Vert \nabla u\Vert _{\mathcal{L}_{q}}}\dint_{\Omega
}(x-f_{0})\cdot \nabla u~dm \\
&\leq &\left\Vert x-f_{0}\right\Vert _{\mathcal{L}_{p}}\text{ \ \ \ \ \ }%
(u\in W_{0}^{1,q}(\Omega )\backslash \{0\}),
\end{eqnarray*}%
with equality precisely when $\nabla u$ is a negative multiple of $%
\left\Vert x-f_{0}\right\Vert ^{p-2}(x-f_{0})$. It now follows from (\ref%
{Qdef}) and Proposition \ref{Qq}(ii) that $u_{0}$ is a positive multiple of $%
w_{q}$, and from part (i) and (\ref{Dpperpb}) that (\ref{QqDp}) holds.
\end{proof}

\bigskip

The next result shows that (\ref{QqDp}) also holds when $p=1$. Inequality (%
\ref{LAP}) will follow from (\ref{LDeq}) in view of (\ref{Lineq}).

\begin{proposition}
\label{LD}If $\Omega \subset \mathbb{R}^{N}$ is a smoothly bounded domain
and $p\in \lbrack 1,\infty )$, then 
\begin{equation}
\lambda _{D_{p}}(\Omega )=Q_{q}(\Omega )=\dfrac{-1}{\Vert \nabla w_{q}\Vert
_{\mathcal{L}_{q}}}\int_{\Omega }(x-f)\cdot \nabla w_{q}~dm\text{ \ \ \ }%
(f\in D_{p}(\Omega )).  \label{LDeq}
\end{equation}
\end{proposition}

\begin{proof}
We know from Theorem 1 of \cite{Kaw} that $w_{q}\rightarrow w_{\infty }$
uniformly on $\Omega $ as $q\rightarrow \infty $. Since the function $%
p\mapsto \left( m(\Omega )\right) ^{-1/p}\lambda _{D_{p}}(\Omega )$ is
increasing, we see from Propositions \ref{Dp}(iv) and \ref{Qq}(iii) that%
\begin{eqnarray}
\lambda _{D_{1}}(\Omega ) &\leq &\left( m(\Omega )\right) ^{1-1/p}\lambda
_{D_{p}}(\Omega )=N\left( m(\Omega )\right) ^{1-1/p}\left( \int_{\Omega
}w_{q}~dm\right) ^{1/p}  \notag \\
&\rightarrow &N\int_{\Omega }w_{\infty }~dm\text{ \ \ \ }(p\rightarrow 1) 
\notag \\
&=&Q_{\infty }(\Omega ).  \label{ineq1}
\end{eqnarray}%
For large $k\in \mathbb{N}$ let $v_{k}$ be a mollification of $(w_{\infty
}-k^{-1})^{+}$ that belongs to $C_{c}^{\infty }(\Omega )$. Since $\nabla
v_{k}\in D_{1}(\Omega )^{\perp }$ by (\ref{defDp}), and $(\nabla v_{k})$ is
boundedly convergent almost everywhere to $\nabla w_{\infty }$, we see that $%
\nabla w_{\infty }\in D_{1}(\Omega )^{\perp }$. Thus, by the divergence
theorem,%
\begin{eqnarray*}
Q_{\infty }(\Omega ) &=&N\int w_{\infty }~dm=-\int_{\Omega }x\cdot \nabla
w_{\infty }~dm \\
&=&-\int_{\Omega }(x-f)\cdot \nabla w_{\infty }~dm\leq \left\Vert
x-f\right\Vert _{\mathcal{L}_{1}}\text{ \ \ \ }(f\in D_{1}(\Omega )).
\end{eqnarray*}%
Hence $Q_{\infty }(\Omega )\leq \lambda _{D_{1}}(\Omega )$, and (\ref{LDeq})
follows in view of (\ref{ineq1}).
\end{proof}

\bigskip

We note that 
\begin{equation}
\lambda _{B_{p}}(\Omega )=\inf \{\Vert \nabla u\Vert _{\mathcal{L}_{p}}:u\in
W^{1,p}(\Omega )\text{ \ and \ }\Delta u=Nm\text{ in }\Omega \}.
\label{altLBN}
\end{equation}

\begin{proposition}
\label{Bpm}Let $p\in \lbrack 1,\infty )$. \newline
(i) There exists $f\in B_{p}(\Omega )$ such that $\lambda _{B_{p}}(\Omega
)=\left\Vert x-f\right\Vert _{\mathcal{L}_{p}}$; equivalently, there exists $%
u_{0}\in W^{1,p}(\Omega )$ such that $\Delta u_{0}=Nm$ in $\Omega $ and $%
\lambda _{B_{p}}(\Omega )=\Vert \nabla u_{0}\Vert _{\mathcal{L}_{p}}$.%
\newline
(ii) The function $u_{0}$ satisfies $\Vert \nabla u_{0}\Vert ^{p-2}\nabla
u_{0}\in B_{p}(\Omega )^{\perp }$.\newline
(iii) The function $f_{0}=\Vert \nabla u_{0}\Vert ^{p-2}\nabla u_{0}\in 
\mathcal{L}_{q}$ satisfies $\lambda _{B_{p}}(\Omega )=\left( \int_{\Omega
}f_{0}\cdot x~dm\right) /\Vert f_{0}\Vert _{\mathcal{L}_{q}}.$\newline
(iv) The function $u_{0}$ is unique up to an additive constant, and $\nabla
u_{0}$ is uniquely determined by the properties 
\begin{equation}
\left\{ 
\begin{array}{l}
\Vert \nabla u_{0}\Vert ^{p-2}\nabla u_{0}\in B_{p}(\Omega )^{\perp } \\ 
\Delta u_{0}=Nm\text{ in }\Omega%
\end{array}%
\right. .  \label{uchar}
\end{equation}
\end{proposition}

\begin{proof}
(i) We can choose a minimizing sequence $(u_{j})$ for (\ref{altLBN}), where $%
\int_{\Omega }u_{j}~dm=0$ for each $j$. By Poincar\'{e}'s inequality $%
(u_{j}) $ is bounded in $W^{1,p}(\Omega )$, and by the Rellich-Kondrachov
theorem we can arrange that $(u_{j})$ converges strongly in $L^{1}(\Omega )$
to a function $u_{0}$. Since the functions $\{u_{j}:j\geq 0\}$ all have
distributional Laplacian equal to $Nm$, we can choose smooth representatives
of these functions and arrange that $u_{j}\rightarrow u_{0}$ and $\partial
u_{j}/\partial x_{i}\rightarrow \partial u_{0}/\partial x_{i}$ locally
uniformly on $\Omega $ for each $i$. Now%
\begin{equation*}
\left\vert \int_{\Omega }\nabla u_{j}\cdot \phi ~dm\right\vert \leq
\left\Vert \nabla u_{j}\right\Vert _{\mathcal{L}_{p}}\left\Vert \phi
\right\Vert _{\mathcal{L}_{q}}\text{ \ \ \ }(\phi \in \left( C_{c}^{\infty
}(\Omega )\right) ^{N};j\geq 1),
\end{equation*}%
so we can let $j\rightarrow \infty $ and use the density of $C_{c}^{\infty
}(\Omega )$ in $L^{q}(\Omega )$ to see that $\left\Vert \nabla
u_{0}\right\Vert _{\mathcal{L}_{p}}\leq \lambda _{B_{p}}(\Omega )$. (When $%
p=1$ and so $q=\infty $, we instead use the fact that, for any $g\in \left(
L^{\infty }(\Omega )\right) ^{N}$, there is a sequence $(\phi _{n})$ in $%
\left( C_{c}^{\infty }(\Omega )\right) ^{N}$ that converges pointwise almost
everywhere to $g$ on $\Omega $ and satisfies $\sup_{\Omega }\left\Vert \phi
_{n}\right\Vert \leq \mathrm{ess}\sup_{\Omega }\left\Vert g\right\Vert $ for
all $n$.) Similarly, $u_{0}\in L^{p}(\Omega )$, so $u_{0}\in W^{1,p}(\Omega
) $ and $u_{0}$ is a minimizer for (\ref{altLBN}).

(ii) Given any $h\in W^{1,p}(\Omega )\cap C^{2}(\Omega )$ such that $\Delta
h=0$ on $\Omega $, we differentiate $\Vert \nabla (u_{0}+th)\Vert _{\mathcal{%
L}_{p}}^{p}$ with respect to $t$ and then put $t=0$ to see that 
\begin{equation}
\int_{\Omega }\Vert \nabla u_{0}\Vert ^{p-2}\nabla u_{0}\cdot \nabla h~dm=0.
\label{orth}
\end{equation}%
(When $p=1$, we know that $m(\{\Vert \nabla u_{0}||=0\})=0$, and the above
equation still follows by dominated convergence, since $\left\vert
\left\Vert \nabla (u_{0}+th)\right\Vert -\left\Vert \nabla u_{0}\right\Vert
\right\vert /t\leq \left\Vert \nabla h\right\Vert $.) Thus $\Vert \nabla
u_{0}\Vert ^{p-2}\nabla u_{0}\in B_{p}(\Omega )^{\perp }$.

(iii) If we take $h=u_{0}-\left\Vert x\right\Vert ^{2}/2$ in (\ref{orth}),
then we find that%
\begin{equation*}
\int_{\Omega }f_{0}\cdot x~dm=\int_{\Omega }\Vert \nabla u_{0}\Vert ^{p}~dm.
\end{equation*}%
Since $\Vert f_{0}\Vert _{\mathcal{L}_{q}}=\Vert \nabla u_{0}\Vert _{%
\mathcal{L}_{p}}^{p/q}$, we obtain the desired equality.

(iv) In view of parts (i) and (ii) it only remains to check that (\ref{uchar}%
) uniquely determines $u_{0}$ up to a constant. (When $p>1$, the uniqueness
of the gradient $\nabla u_{0}$ also follows from the strict convexity of the 
$\mathcal{L}_{p}$-norm.) To see this, let $v$ be another such function and
consider the harmonic function $v-u_{0}$. It follows from (\ref{orth}) that 
\begin{equation*}
\int_{\Omega }\Vert \nabla u_{0}\Vert ^{p}~dm=\int_{\Omega }\Vert \nabla
u_{0}\Vert ^{p-2}\nabla u_{0}\cdot \nabla v~dm
\end{equation*}%
and%
\begin{equation*}
\int_{\Omega }\Vert \nabla v\Vert ^{p}~dm=\int_{\Omega }\Vert \nabla v\Vert
^{p-2}\nabla v\cdot \nabla u_{0}~dm.
\end{equation*}%
H\"{o}lder's inequality now shows that $\left\Vert \nabla u_{0}\right\Vert _{%
\mathcal{L}_{p}}=\left\Vert \nabla v\right\Vert _{\mathcal{L}_{p}}$, and we
deduce that $\nabla u_{0}\equiv \nabla v$. (If $p=1$, then H\"{o}lder's
inequality is unnecessary.)
\end{proof}

\begin{proposition}
\label{Apm}Let $p\in \lbrack 1,\infty )$. \newline
(i) There exists $f\in A_{p}(\Omega )$ such that $\lambda _{A_{p}}(\Omega
)=\left\Vert x-f\right\Vert _{\mathcal{L}_{p}}$.\newline
(ii) The function $f$ satisfies $\Vert x-f\Vert ^{p-2}(x-f)\in A_{p}(\Omega
)^{\perp }$.\newline
(iii) The function $f_{0}=\Vert x-f\Vert ^{p-2}(x-f)\in \mathcal{L}_{q}$
satisfies $\lambda _{A_{p}}(\Omega )=\left( \int_{\Omega }f_{0}\cdot
x~dm\right) /\Vert f_{0}\Vert _{\mathcal{L}_{q}}.$\newline
(iv) The function $f$ is uniquely determined by the properties 
\begin{equation*}
\left\{ 
\begin{array}{l}
\Vert x-f\Vert ^{p-2}(x-f)\in A_{p}(\Omega )^{\perp } \\ 
\func{div}f=0\text{ and }\func{curl}f=0\text{ in }\Omega%
\end{array}%
\right. .
\end{equation*}
\end{proposition}

\begin{proof}
(i) We choose a sequence $(f^{(j)})$ in $A_{p}(\Omega )$ such that $%
\left\Vert x-f^{(j)}\right\Vert _{\mathcal{L}_{p}}\rightarrow \lambda
_{A_{p}}(\Omega )$. Since $(\left\Vert f^{(j)}\right\Vert _{\mathcal{L}%
_{p}}) $ is bounded and the functions $\left\Vert f^{(j)}\right\Vert $ are
subharmonic (by Theorem 3.4.5 of \cite{AG}), the harmonic co-ordinate
functions $f_{i}^{(j)}$ ($i=1,...,N$) are locally uniformly bounded. Thus,
by taking a subsequence, we can arrange that $(f^{(j)})$ converges locally
uniformly to some function $f$ satisfying $\func{div}f=0$ and $\func{curl}%
f=0 $ on $\Omega $. Since%
\begin{equation*}
\left\vert \int (x-f^{(j)})\cdot \phi ~dm\right\vert \leq \left\Vert
x-f^{(j)}\right\Vert _{\mathcal{L}_{p}}\left\Vert \phi \right\Vert _{%
\mathcal{L}_{q}}\text{ \ \ \ }(\phi \in (C_{c}^{\infty }(\Omega ))^{N}),
\end{equation*}%
we can let $j\rightarrow \infty $ and use the density of $C_{c}^{\infty
}(\Omega )$ in $L^{q}(\Omega )$ to see that $\left\Vert x-f\right\Vert _{%
\mathcal{L}_{p}}\leq \lambda _{A_{p}}(\Omega )$. (When $p=1$ we make the
same adjustments to this argument as in the proof of Proposition \ref{Bpm}%
(i).) The reverse inequality is trivial.

(ii) - (iv) The arguments are analogous to those given for the previous
proposition.
\end{proof}

\section{Proofs of Theorems \protect\ref{main1} and \protect\ref{main2}}

As noted previously, inequality (\ref{LAP}) follows from (\ref{LDeq}) and (%
\ref{Lineq}). In this section we will complete the proofs of Theorem 2
(except where $p=2$) and Theorem 3. In view of (\ref{Lineq}) and Proposition %
\ref{LD}, Theorem \ref{main2} is a consequence of the result below.

\begin{theorem}
\label{LB}If $\Omega \subset \mathbb{R}^{N}$ is a smoothly bounded domain
and $p\in \lbrack 1,2]$, then%
\begin{equation}
\lambda _{B_{p}}(\Omega )\leq Q_{q}(B(r_{\Omega })).  \label{LBineq}
\end{equation}%
Further, equality holds if and only if $\Omega $ is a ball.
\end{theorem}

\begin{proof}
Let $u$ be the Green potential satisfying $\Delta u=N$ on $\Omega $ and $u=0$
on $\partial \Omega $. Next, let $w(x)=(\Vert x\Vert ^{2}-r_{\Omega }^{2})/2$%
, so that $\Delta w=N$ in $B(r_{\Omega })$ and $w=0$ on $\partial
B(r_{\Omega })$. We make use of a result of Talenti \cite{Tal} concerning
spherical rearrangements. Theorem 1(v) of that paper tells us that, provided 
$p\leq 2$, we have $\Vert \nabla u\Vert _{\mathcal{L}_{p}(\Omega )}\leq
\Vert \nabla w\Vert _{\mathcal{L}_{p}(B(r_{\Omega }))}$. Hence 
\begin{eqnarray}
\lambda _{B_{p}}(\Omega ) &\leq &\Vert \nabla u\Vert _{\mathcal{L}_{p}}\leq
\left\{ \int_{B(r_{\Omega })}\Vert x\Vert ^{p}dm\right\} ^{1/p}  \notag \\
&=&\left\{ \frac{N}{N+p}m(B)r_{\Omega }^{p+N}\right\}
^{1/p}=Q_{q}(B(r_{\Omega })),  \label{x}
\end{eqnarray}%
by (\ref{altLBN}) and then Proposition \ref{Qq}(iv).

Finally, if $\Vert \nabla u\Vert _{\mathcal{L}_{p}(\Omega )}=\Vert \nabla
w\Vert _{\mathcal{L}_{p}(B(r_{\Omega }))}$, {then }Propositions 3.2.1 and
3.2.2 of Kesavan \cite{Ke} tell us that $\Omega $ must be a ball.
\end{proof}

\begin{lemma}
\label{ann}Let $p\in \lbrack 1,\infty )$. If $\Omega $ is either a ball or
an annular region, then 
\begin{equation*}
\lambda _{B_{p}}(\Omega )=\lambda _{A_{p}}(\Omega )=\lambda _{D_{p}}(\Omega
)=Q_{q}(\Omega ).
\end{equation*}
\end{lemma}

\begin{proof}
In view of (\ref{Lineq}) and Proposition \ref{LD} it is enough to show that $%
\lambda _{B_{p}}(\Omega )\leq Q_{q}(\Omega )$ when $\Omega $ is either a
ball or an annular region. If $\Omega =B(r)$, then (cf. (\ref{x}))%
\begin{equation*}
\lambda _{B_{p}}(B(r))\leq \left\Vert x\right\Vert _{\mathcal{L}%
_{p}}=Q_{q}(B(r)).
\end{equation*}%
Thus it remains to consider the case where $\Omega =B(R)\setminus \overline{%
B(r)}$ and $0<r<R$.

If $p>1$, then it follows from spherical symmetry that there exists $v\in
C^{1}(\overline{\Omega })$ such that $\nabla v=\left\Vert \nabla
w_{q}\right\Vert ^{q-2}\nabla w_{q}$. Writing $w=-Nv$, we see that $\Delta
w=-N\Delta _{q}w_{q}=N$ and so, by (\ref{altLBN}) and Proposition \ref{Qq}%
(iii), 
\begin{equation*}
\lambda _{B_{p}}(\Omega )\leq \left\Vert \nabla w\right\Vert _{\mathcal{L}%
_{p}}=N\left\Vert \nabla v\right\Vert _{\mathcal{L}_{p}}=N\left\Vert \nabla
w_{q}\right\Vert _{\mathcal{L}_{q}}^{q-1}=Q_{q}(\Omega ),
\end{equation*}%
as required.

Now suppose that $p=1$. By Proposition \ref{Qq}(iii) again, 
\begin{eqnarray}
Q_{\infty }(\Omega ) &=&N\int_{\Omega }w_{\infty }(x)dm  \notag \\
&=&N\left( \int_{\left\{ r<\left\Vert x\right\Vert <(R+r)/2\right\}
}(\left\Vert x\right\Vert -r)dm(x)+\int_{\left\{ (R+r)/2<\left\Vert
x\right\Vert <R\right\} }(R-\left\Vert x\right\Vert )dm(x)\right)  \notag \\
&=&\frac{Nm(B)}{N+1}\left( R^{N+1}+r^{N+1}-\frac{\left( R+r\right) ^{N+1}}{%
2^{N}}\right) .  \label{Qinf}
\end{eqnarray}%
If we define 
\begin{equation*}
u(x)=\left\{ 
\begin{array}{cc}
\dfrac{\Vert x\Vert ^{2}}{2}+\dfrac{1}{N-2}\left( \dfrac{R+r}{2}\right)
^{N}\Vert x\Vert ^{2-N} & (N\geq 3) \\ 
\dfrac{\Vert x\Vert ^{2}}{2}-\left( \dfrac{R+r}{2}\right) ^{2}\log \Vert
x\Vert & (N=2)%
\end{array}%
\right. ,
\end{equation*}%
then 
\begin{eqnarray*}
\lambda _{B_{1}}(\Omega ) &\leq &\Vert \nabla u\Vert _{\mathcal{L}%
_{1}}=\int_{\Omega }\left\vert \Vert x\Vert -\left( \frac{R+r}{2}\right)
^{N}\Vert x\Vert ^{1-N}\right\vert dm \\
&=&\int_{\left\{ r<\left\Vert x\right\Vert <(R+r)/2\right\} }\left( \left( 
\frac{R+r}{2}\right) ^{N}\Vert x\Vert ^{1-N}-\Vert x\Vert \right) dm \\
&&+\int_{\left\{ (R+r)/2<\left\Vert x\right\Vert <R\right\} }\left( \Vert
x\Vert -\left( \frac{R+r}{2}\right) ^{N}\Vert x\Vert ^{1-N}\right) dm \\
&=&\frac{Nm(B)}{N+1}\left( R^{N+1}+r^{N+1}-\frac{\left( R+r\right) ^{N+1}}{%
2^{N}}\right) =Q_{\infty }(\Omega ),
\end{eqnarray*}%
by (\ref{Qinf}).
\end{proof}

\bigskip

\begin{proposition}
\label{conserv}Let $p\in \lbrack 1,\infty )$. If there exists $f\in
A_{p}(\Omega )$ satisfying $\left\Vert x-f\right\Vert _{\mathcal{L}%
_{p}}=Q_{q}(\Omega )$, then $f\in B_{p}(\Omega )$.
\end{proposition}

\begin{proof}
First suppose that $p>1$, so that $q<\infty $. By (\ref{Lineq}) and
Proposition \ref{LD}, 
\begin{equation*}
\left\Vert x-f\right\Vert _{\mathcal{L}_{p}}=\lambda _{D_{p}}(\Omega )=\frac{%
-1}{\left\Vert \nabla w_{q}\right\Vert _{\mathcal{L}_{q}}}\int_{\Omega
}(x-f)\cdot \nabla w_{q}~dm.
\end{equation*}%
Since 
\begin{equation}
-\int_{\Omega }(x-f)\cdot \nabla w_{q}~dm=\left\Vert x-f\right\Vert _{%
\mathcal{L}_{p}}\left\Vert \nabla w_{q}\right\Vert _{\mathcal{L}_{q}},
\label{eq}
\end{equation}%
the equality case of H\"{o}lder's inequality implies that $x-f=c\left\Vert
\nabla w_{q}\right\Vert ^{q-2}\nabla w_{q}$ on $\Omega $ for some constant $%
c $. Hence $x-f$ has a continuous extension to $\overline{\Omega }$, and on $%
\partial \Omega $ it is normal to $\partial \Omega $. We now choose a
function $v\in C^{1}(\mathbb{R}^{N}\backslash \Omega )$ such that $v=0$ and $%
\nabla v=x-f$ on $\partial \Omega $. (Such a function exists by \cite{GG},
for example.) Thus we obtain a continuous extension of $f$ to $\mathbb{R}%
^{N} $ by defining it to be $x-\nabla v$ on $\mathbb{R}^{N}\backslash \Omega 
$.

We claim that this extended function, which we also denote by $f$, is
curl-free in the sense of distributions. By using a partition of unity it is
enough to show that, for some $\delta >0$, 
\begin{equation}
\int \left( f_{i}\frac{\partial \phi }{\partial x_{j}}-f_{j}\frac{\partial
\phi }{\partial x_{i}}\right) dm=0\text{ \ \ \ }(i\neq j)  \label{cur}
\end{equation}%
whenever $\phi \in C^{\infty }(\mathbb{R}^{N})$ and $\mathrm{diam}(\mathrm{%
supp}(\phi ))<\delta $. This equation trivially holds when $\mathrm{supp}%
(\phi )\cap \partial \Omega =\emptyset $, so it is enough to consider the
case where 
\begin{equation*}
\mathrm{supp}(\phi )\subset K:=\underset{i=1}{\overset{N}{\times }}%
(y_{i}-r,y_{i}+r)
\end{equation*}%
for some $y\in \partial \Omega $ and $r>0$.

Without loss of generality we may assume that 
\begin{equation*}
K\cap \partial \Omega =\left\{ \left(
x_{1},...,x_{N-1},g(x_{1},...,x_{N-1})\right) :x_{i}\in (y_{i}-r,y_{i}+r)%
\text{ whenever }i<N\right\}
\end{equation*}%
for some smooth function $g$. If $i<N$ and the co-ordinates $x_{j}$ ($j\neq
i,N$) are fixed, then \ \ \ 
\begin{eqnarray*}
&&\int_{y_{N}-r}^{y_{N}+r}\int_{y_{i}-r}^{y_{i}+r}\left( f_{i}\frac{\partial
\phi }{\partial x_{N}}-f_{N}\frac{\partial \phi }{\partial x_{i}}\right)
dx_{i}dx_{N} \\
&=&\left( \int_{D_{1}}+\int_{D_{2}}\right) \left( f_{i}\frac{\partial \phi }{%
\partial x_{N}}-f_{N}\frac{\partial \phi }{\partial x_{i}}\right)
dA(x_{i},x_{N}),
\end{eqnarray*}%
where $D_{1},D_{2}$ are the components of $\{(x_{i},x_{N}):(x_{1},...,x_{N})%
\in K\backslash \partial \Omega \}$ and $A$ denotes two-dimensional measure.
Two applications of Green's theorem, together with the fact that $\partial
f_{i}/\partial x_{N}=\partial f_{N}/\partial x_{i}$ on $\mathbb{R}%
^{N}\backslash \partial \Omega $, show that this latter integral expression
reduces to self-cancelling terms along the common boundary curve of $%
D_{1},D_{2}$. Hence (\ref{cur}) holds when $j=N$. If $j\neq N$, we apply a
small rotation in the $(x_{j},x_{N})$-plane to see similarly that%
\begin{equation*}
\int \left\{ f_{i}\left( \frac{\partial \phi }{\partial x_{N}}\cos \theta +%
\frac{\partial \phi }{\partial x_{j}}\sin \theta \right) -\left( f_{N}\cos
\theta +f_{j}\sin \theta \right) \frac{\partial \phi }{\partial x_{i}}%
\right\} dm=0,
\end{equation*}%
whence (\ref{cur}) again follows.

We now use a rotationally invariant smoothing kernel $\psi _{\varepsilon }$
supported by a ball of radius $\varepsilon $ to obtain a mollification $%
f^{\varepsilon }$ of $f$, which is also curl-free since 
\begin{equation*}
\frac{\partial }{\partial x_{j}}\psi _{\varepsilon }(x-y)=-\frac{\partial }{%
\partial y_{j}}\psi _{\varepsilon }(x-y).
\end{equation*}%
Further, since each component $f_{i}$ of $f$ is harmonic in $\Omega $, the
functions $f^{\varepsilon }$ and $f$ are equal on the set $\{x:\mathrm{dist}%
(x,\mathbb{R}^{N}\backslash \Omega )>\varepsilon \}$. Hence line integrals
of $f$ in $\Omega $ are path independent, so $f$ is of the form $\nabla v$,
where $\Delta v=\nabla \cdot f=0$, and thus $f\in B_{p}(\Omega )$.

Finally, if $p=1$, then (\ref{eq}) still holds, and now shows that $%
x-f=-\left\Vert x-f\right\Vert \nabla w_{\infty }$ on $\Omega $. We can thus
apply the above argument to $\Omega _{\eta }=\{x:\mathrm{dist}(x,\mathbb{R}%
^{N}\backslash \Omega )>\eta \}$ to deduce that $f\in B_{1}(\Omega _{\eta })$
for arbitrarily small $\eta >0$, and so $f\in B_{1}(\Omega )$.
\end{proof}

\bigskip

We now consider the overdetermined problem%
\begin{equation}
\left\{ 
\begin{array}{ll}
\Delta v=1 & \text{ in }\Omega \\ 
v=c_{i}\text{ \ and \ }\dfrac{\partial v}{\partial n}=a_{i} & \text{ on }%
\Gamma _{i}%
\end{array}%
\right. ,  \label{od}
\end{equation}%
where $n$ denotes the exterior unit normal, $a_{i},c_{i}\in \mathbb{R}$ $%
(i=0,...,j)$ and $\{\Gamma _{i}\}$ are the components of $\partial \Omega $.
(We use $\Gamma _{0}$ for the outer boundary component.) The following
theorem, which generalizes earlier work of Serrin \cite{Ser}, is contained
in Theorem 2 of Sirakov \cite{Sir}.

\begin{theorem}
\label{KS}Let $c_{0}=0$, $a_{0}\geq 0$, and $c_{i}<0$, $a_{i}\leq 0$ $%
(i=1,...,j)$. Then there exists $v\in C^{2}(\overline{\Omega })$ satisfying (%
\ref{od}) if and only if $\Omega $ is a ball or an annular region. In either
case, $v$ is a radial function.
\end{theorem}

The case $p\neq 2$ of the equality statement in Theorem \ref{main1} is
established in the next result. The case where $p=2$ will be addressed in
Section \ref{S4}.

\begin{theorem}
Let $p\in \lbrack 1,\infty )$, where $p\neq 2$. Then $\lambda
_{A_{p}}(\Omega )=Q_{q}(\Omega )$ if and only if $\Omega $ is either a ball
or an annular region.
\end{theorem}

\begin{proof}
For the \textquotedblleft if\textquotedblright\ part we refer to Lemma \ref%
{ann}. For the \textquotedblleft only if\textquotedblright\ part it is
enough, given Propositions \ref{Apm}\ and \ref{conserv}, to show that, if
there exists $v\in W^{1,p}(\Omega )$ such that $\Delta v=Nm$ and $\left\Vert
\nabla v\right\Vert _{\mathcal{L}_{p}}=Q_{q}(\Omega )$, then $\Omega $ is
either a ball or an annular region.

If $p>1$, then we see from Proposition \ref{Qq} that 
\begin{equation*}
Q_{q}(\Omega )=\frac{N\int_{\Omega }w_{q}~dm}{\left\Vert \nabla
w_{q}\right\Vert _{\mathcal{L}_{q}}}=\frac{\int_{\Omega }w_{q}\Delta v~dm}{%
\left\Vert \nabla w_{q}\right\Vert _{\mathcal{L}_{q}}}=-\frac{\int_{\Omega
}\nabla w_{q}\cdot \nabla v~dm}{\left\Vert \nabla w_{q}\right\Vert _{%
\mathcal{L}_{q}}},
\end{equation*}%
where the last equality can be justified using the facts that $w_{q}\in
W_{0}^{1,q}(\Omega )$ and that $C_{c}^{\infty }(\Omega )$ is dense in $%
W_{0}^{1,q}(\Omega )$\textbf{.} By H\"{o}lder's inequality, 
\begin{equation*}
\left\vert \int_{\Omega }\nabla w_{q}\cdot \nabla v~dm\right\vert \leq
\left\Vert \nabla v\right\Vert _{\mathcal{L}_{p}}\left\Vert \nabla
w_{q}\right\Vert _{\mathcal{L}_{q}},
\end{equation*}%
where equality occurs if and only if $\nabla w_{q},\nabla v$ are always
parallel, $\nabla w_{q}\cdot \nabla v$ does not change sign, and $\left\Vert
\nabla w_{q}\right\Vert ^{q}=c\left\Vert \nabla v\right\Vert ^{p}$ in $%
\Omega $ for some constant $c>0$. Further, $\nabla w_{q}\neq 0$ on $\partial
\Omega $ by Hopf's lemma (see Theorem 5.5.1 of \cite{PS}). Thus the equality 
$Q_{q}(\Omega )=\left\Vert \nabla v\right\Vert _{\mathcal{L}_{p}}$ implies
that each component of any level surface of $v$ is also a component of a
level surface of $w_{q}$. Hence, for each component $\Gamma _{i}$ of $%
\partial \Omega $, there is a function $g_{i}$ such that $v=g_{i}\circ w_{q}$
near $\Gamma _{i}$. Further, $\nabla v=g_{i}^{\prime }(w_{q})\nabla w_{q}$,
so 
\begin{equation*}
\left\Vert \nabla v\right\Vert =\left\vert g_{i}^{\prime }(w_{q})\right\vert
\left\Vert \nabla w_{q}\right\Vert =c^{1/q}\left\vert g_{i}^{\prime
}(w_{q})\right\vert \left\Vert \nabla v\right\Vert ^{p/q}.
\end{equation*}%
Since $p\neq 2$, we have $p\neq q$ and so $\left\Vert \nabla v\right\Vert
=c^{1/(q-p)}\left\vert g_{i}^{\prime }(w_{q})\right\vert ^{q/(q-p)}$. Thus $%
\left\Vert \nabla v\right\Vert $ is constant on each component of a level
surface of $w_{q}$ (which is also a level surface of $v$).

Since $\nabla w_{q}\cdot n<0$ on $\partial \Omega $ and $\nabla w_{q}\cdot
\nabla v$ does not change sign, we can apply the divergence theorem to $%
\nabla v$ to see that $\nabla w_{q}\cdot \nabla v<0$ near $\partial \Omega $
and hence $\nabla v\cdot n>0$ on $\Gamma _{0}$. Now let $\varepsilon >0$ be
small and let $\Omega _{\varepsilon }$ be the component of $%
\{0<w_{q}<\varepsilon \}$ which has $\Gamma _{0}$ as a boundary component.
Since $w_{q}\in C^{1}(\overline{\Omega })$, $\Vert \nabla w_{q}\Vert
^{q}=c\Vert \nabla v\Vert ^{p}$ and $\Vert \nabla v\Vert
=c^{1/(q-p)}|g_{0}^{\prime }(w_{q})|^{q/(q-p)}$, it follows that $\left\vert
g_{0}^{\prime }(0)\right\vert <\infty $ and certainly $\left\vert
g_{0}(0)\right\vert <\infty $. Thus $v$ has a (finite) constant value on
each component of $\partial \Omega _{\varepsilon }$. Since $\Delta v=Nm$ on $%
\Omega $ we conclude that $v\in C^{2}(\overline{\Omega }_{\varepsilon })$
(see Theorem 6.14 of \cite{GT}). Thus we can apply Theorem \ref{KS} to $%
(v-g_{0}(0))/N$ on $\Omega _{\varepsilon }$ to see that $\Gamma _{0}$ is a
sphere and $v$ is a radial function. By the analyticity of $v$, any other
boundary of component of $\Omega $ must be a concentric sphere. Thus $\Omega 
$ is either a ball or an annular region.

The argument for the case $p=1$ is mostly similar. Since $\left\Vert \nabla
w_{\infty }\right\Vert =1$, we have 
\begin{equation*}
Q_{\infty }(\Omega )=N\int_{\Omega }w_{\infty }~dm=-\int_{\Omega }\nabla
w_{\infty }\cdot \nabla v~dm
\end{equation*}%
and 
\begin{equation*}
\left\vert \int_{\Omega }\nabla w_{\infty }\cdot \nabla v~dm\right\vert \leq
\left\Vert \nabla v\right\Vert _{\mathcal{L}_{1}}.
\end{equation*}%
The equality $Q_{\infty }(\Omega )=\left\Vert \nabla v\right\Vert _{\mathcal{%
L}_{1}}$ implies that $\nabla w_{\infty }$ and $\nabla v$ are parallel.
Thus, for each component $\Gamma _{i}$ of $\partial \Omega $, there is a
function $g_{i}$ such that $v=g_{i}\circ w_{\infty }$ near $\Gamma _{i}$. We
no longer claim that this equation holds on $\partial \Omega $. However, as
in the proof of Proposition \ref{conserv}, we can work instead with $\Omega
_{\eta }=\{x:\mathrm{dist}(x,\mathbb{R}^{N}\backslash \Omega )>\eta \}$ for
small $\eta >0$ and now argue as before to conclude that $\Omega $ is either
a ball or an annular region.
\end{proof}

\section{The case where $p=2$\label{S4}}

It follows from Proposition \ref{Bpm}(i) that there exist harmonic functions 
$h$ satisfying $\lambda _{B_{2}}(\Omega )=\Vert x-\nabla h\Vert _{\mathcal{L}%
_{2}}$. We will now identify all such functions. (This was already done in 
\cite{FK} in the case of planar domains.)

\begin{theorem}
\label{main2'} The harmonic functions $h\in W^{1,2}({\Omega })$ which
satisfy $\lambda _{B_{2}}(\Omega )=\Vert x-\nabla h\Vert _{\mathcal{L}_{2}}$
are precisely the functions of the form $H_{\Vert x\Vert ^{2}/2}^{\Omega }+c%
\mathrm{,}$ where $H_{g}^{\Omega }$ is the solution to the Dirichlet problem
on $\Omega $ with boundary data $g$, and $c\in \mathbb{R}$.
\end{theorem}

\begin{proof}
Let $h\in W^{1,2}(\Omega )$ be a harmonic function satisfying $\lambda
_{B_{2}}(\Omega )=\Vert x-\nabla h\Vert _{\mathcal{L}_{2}}$, and let $k\in
C^{1}(\overline{\Omega })\ $be harmonic on $\Omega $. Since the function $%
t\mapsto \Vert x-\nabla (h+tk)\Vert _{2}^{2}$ has a minimum at $t=0$, we see
that 
\begin{equation*}
\int_{\Omega }(x-\nabla h)\cdot \nabla k~dm=0.
\end{equation*}%
Hence, by the divergence theorem, 
\begin{equation}
\int_{\partial \Omega }\left( \frac{\Vert x\Vert ^{2}}{2}-h\right) \frac{%
\partial k}{\partial n}~d\sigma =0,  \label{bdyc}
\end{equation}%
where $\sigma $ denotes surface area measure. Since we can solve the Neumann
problem 
\begin{equation*}
\left\{ 
\begin{array}{ll}
\Delta k=0 & \text{ in }\Omega \\ 
\dfrac{\partial k}{\partial n}=\phi & \text{ on }\partial \Omega%
\end{array}%
\right.
\end{equation*}%
for any smooth function $\phi $ satisfying $\int_{\partial \Omega }\phi
d\sigma =0$, we see from (\ref{bdyc}) that $\left\Vert x\right\Vert
^{2}/2-h(x)\mathrm{\ }$is constant on $\partial \Omega .$
\end{proof}

\bigskip

The \textit{torsional rigidity of }$\Omega $ is defined by%
\begin{equation*}
\rho (\Omega )=\int_{\Omega }\Vert \nabla v\Vert ^{2}dm,
\end{equation*}%
where $v$ is the solution to the Dirichlet problem%
\begin{equation}
\left\{ 
\begin{array}{ll}
-\Delta v=N & \text{ in }\Omega \\ 
v=0 & \text{ on }\Gamma _{0} \\ 
v=c_{i} & \text{ on }\Gamma _{i}\text{ for }i=1,\ldots ,j%
\end{array}%
\right. ;  \label{ci}
\end{equation}%
here $\Gamma _{0}$ is again the boundary of the unbounded component of $%
\mathbb{R}^{N}\backslash \Omega $, while $G_{1},G_{2},\ldots ,G_{j}$ are the
bounded components of $\Omega ^{c}$ with boundaries $\Gamma _{1},\Gamma
_{2},\ldots ,\Gamma _{j}$ and the constants $c_{i}$ are chosen so that 
\begin{equation}
\int_{\Gamma _{i}}\frac{\partial v}{\partial n}d\sigma =2m(G_{i})\text{ \ \ }%
(i=1,2,\ldots ,j).  \label{ci2}
\end{equation}%
From Proposition \ref{Qq}(iii) we see that 
\begin{equation}
Q_{2}(\Omega )=\sqrt{\rho (\Omega )}  \label{tor}
\end{equation}%
when $\mathbb{R}^{N}\backslash \Omega $ has no bounded components.\bigskip

Theorem \ref{Cor} is contained in the result below.

\begin{theorem}
\label{p2}If $\Omega \subset \mathbb{R}^{N}$ is a smoothly bounded domain,
then%
\begin{equation}
\lambda _{B_{2}}(\Omega )=\lambda _{A_{2}}(\Omega )=\lambda _{D_{2}}(\Omega
)=Q_{2}(\Omega ).  \label{L2eq}
\end{equation}%
Further, these quantities are equal to $\sqrt{\rho (\Omega )}$ if and only
if $\mathbb{R}^{N}\backslash \Omega $ is connected.
\end{theorem}

\begin{proof}[Proof of Theorem \protect\ref{p2}]
Let $u(x)=$ $H_{\Vert x\Vert ^{2}/2}^{\Omega }-\left\Vert x\right\Vert
^{2}/2 $. By Theorem \ref{main2'},%
\begin{equation*}
\left( \lambda _{B_{2}}(\Omega )\right) ^{2}=\int_{\Omega }\Vert x-\nabla
H_{\Vert x\Vert ^{2}/2}^{\Omega }\Vert ^{2}dx=\int_{\Omega }\Vert \nabla
u\Vert ^{2}dx=N\int_{\Omega }u~dx,
\end{equation*}%
where for the last step we applied the divergence theorem and noted that $%
u=0 $ on $\partial \Omega $. Hence%
\begin{equation*}
\lambda _{B_{2}}(\Omega )=\Vert \nabla u\Vert _{\mathcal{L}_{2}}=\frac{%
N\int_{\Omega }u~dm}{\Vert \nabla u\Vert _{\mathcal{L}_{2}}}\leq
Q_{2}(\Omega ).
\end{equation*}%
Equation (\ref{L2eq}) now follows from Proposition \ref{LD} and (\ref{Lineq}%
).

We know from (\ref{tor}) that $Q_{2}(\Omega )=\sqrt{\rho (\Omega )}$ if $%
\mathbb{R}^{N}\backslash \Omega $ is connected. Conversely, suppose that $%
\mathbb{R}^{N}\backslash \Omega $ is not connected, and let $c_{k}=\min
\{c_{1},...,c_{j}\}$. If $c_{k}\leq 0$, then the Hopf boundary point lemma
(see Section 6.4.2 of \cite{Ev}) would tell us that $\partial v/\partial n<0$
on $\Gamma _{k}$, which contradicts (\ref{ci2}). Thus $c_{i}>0$ $(i=1,...,j)$
in (\ref{ci}), so $v$ cannot be a multiple of $w_{2}$, and it now follows
from Proposition \ref{Qq}\ that $Q_{2}(\Omega )>\sqrt{\rho (\Omega )}$.
\end{proof}


\bigskip

\begin{thebibliography}{99}
\bibitem{ABKT} A. Abanov, C. B\'{e}n\'{e}teau, D. Khavinson, and R.
Teodorescu, \textquotedblleft A free boundary problem associated with the
isoperimetric inequality\textquotedblright , \textit{J. Anal. Math., }to
appear; \textit{arXiv:}1601.03885.

\bibitem{AG} D. H. Armitage and S. J. Gardiner, \textit{Classical potential
theory}. Springer, London, 2001.

\bibitem{BFL} S. R. Bell, T. Ferguson and E. Lundberg, \textquotedblleft
Self-commutators of Toeplitz operators and isoperimetric
inequalities\textquotedblright , Math. Proc. R. Ir. Acad. 114A (2014),
115--133.

\bibitem{BK} C. B\'{e}n\'{e}teau and D. Khavinson, \textquotedblleft The
isoperimetric inequality via approximation theory and free boundary
problems\textquotedblright , \textit{Comput. Methods Funct. Theory} 6
(2006), 253--274.

\bibitem{BK2} C. B\'{e}n\'{e}teau and D. Khavinson, \textquotedblleft
Selected problems in classical function theory\textquotedblright , \textit{%
Invariant subspaces of the shift operator}, pp.255--265, Contemp. Math.,
638, Amer. Math. Soc., Providence, RI, 2015.

\bibitem{BdPV} L. Brasco, G. De Philippis and B. Velichkov,
\textquotedblleft Faber-Krahn inequalities in sharp quantitative
form\textquotedblright , \textit{Duke Math. J. }164 (2015), 1777--1831.

\bibitem{Ev} L. C. Evans, \textit{Partial Differential Equations}. 2nd
Edition. Amer. Math. Soc., Providence, RI, 2010.

\bibitem{FK} M. Fleeman and D. Khavinson, \textquotedblleft Approximating $%
\bar{z}$ in the Bergman Space\textquotedblright , \textit{Recent progress on
operator theory and approximation in spaces of analytic functions,}
pp.79--90, Contemp. Math., 679, Amer. Math. Soc., Providence, RI, 2016.

\bibitem{FL} M. Fleeman and E. Lundberg, \textquotedblleft The Bergman
analytic content of planar domains\textquotedblright , \textit{Comput.
Methods Funct. Theory }17 (2017), 369--379.

\bibitem{FS} M. Fleeman and B. Simanek, \textquotedblleft Torsional rigidity
and Bergman analytic content of simply connected domains\textquotedblright , 
\textit{arXiv:}1704.01997v1 (2017).

\bibitem{GaKh} T. W. Gamelin and D. Khavinson, \textquotedblleft The
isoperimetric inequality and rational approximation\textquotedblright , 
\textit{Amer. Math. Monthly }96 (1989), 18--30.

\bibitem{GGS} S. J. Gardiner, M. Ghergu and T. Sj\"{o}din, \textquotedblleft
Analytic Content and the Isoperimetric Inequality in Higher
Dimensions\textquotedblright ,\textit{\ J. Funct. Anal.} 275 (2018),
2284--2298.

\bibitem{GG} S. J. Gardiner and A. Gustafsson, \textquotedblleft Smooth
potentials with prescribed boundary behaviour\textquotedblright , \textit{%
Publ. Mat. }48 (2004), 241--249.

\bibitem{GT} D. Gilbarg and N. S. Trudinger, \textit{Elliptic partial
differential equations of second order.} Springer, Berlin, 2001.

\bibitem{GuKh} Z. Guadarrama and D. Khavinson, \textquotedblleft
Approximating $\overline{z}$ in Hardy and Bergman norms\textquotedblright , 
\textit{Banach spaces of analytic functions}, pp.43--61, Contemp. Math.,
454, Amer. Math. Soc., Providence, RI, 2008.

\bibitem{GK} B. Gustafsson and D. Khavinson, \textquotedblleft On
approximation by harmonic vector fields\textquotedblright , \textit{Houston
J. Math.} 20 (1994), 75--92.

\bibitem{Kaw} B. Kawohl, \textquotedblleft On a family of torsional creep
problems\textquotedblright , \textit{J. Reine Angew. Math. }410 (1990),
1--22.

\bibitem{Ke} S. Kesavan, \textit{Symmetrization and Applications}, World
Scientific, Hackensack, NJ, 2006.

\bibitem{Lie} G. Lieberman, \textquotedblleft Boundary regularity for
solutions of degenerate elliptic equations\textquotedblright , \textit{%
Nonlinear Anal.} 12 (1988), 1203-1219.

\bibitem{PS} P. Pucci and J. Serrin, \textit{The maximum principle}. Birkh%
\"{a}user, Basel, 2007.

\bibitem{Ser} J. Serrin, \textquotedblleft A symmetry problem in potential
theory\textquotedblright , \textit{Arch. Rational Mech. Anal.} 43 (1971),
304--318.

\bibitem{Sir} B. Sirakov, \textquotedblleft Symmetry for exterior elliptic
problems and two conjectures in potential theory\textquotedblright , \textit{%
Ann. Inst. H. Poincar\'{e} Anal. Non Lin\'{e}aire} 18 (2001), 135--156.

\bibitem{Tal} G. Talenti, \textquotedblleft Elliptic equations and
rearrangements\textquotedblright , \textit{Ann. Scuola Norm. Sup. Pisa Cl.
Sci. } (4) 3 (1976), 697--718.

\bigskip

\noindent \textit{Stephen J. Gardiner and Marius Ghergu}

\noindent School of Mathematics and Statistics

\noindent University College Dublin

\noindent Dublin 4, Ireland

\noindent stephen.gardiner@ucd.ie

\noindent marius.ghergu@ucd.ie

\bigskip

\noindent \textit{Tomas Sj\"{o}din}

\noindent Department of Mathematics

\noindent Link\"{o}ping University

\noindent 581 83, Link\"{o}ping

\noindent Sweden

\noindent tomas.sjodin@liu.se
\end{thebibliography}
\end{document}